\documentclass[a4paper,english]{amsart}
\usepackage{setspace}
\usepackage{amssymb}
\usepackage{euscript}
\usepackage[small]{diagrams}
\usepackage[all]{xy}

\theoremstyle{plain}
\newtheorem{thm}{Theorem}[section]

\newtheorem{lem}[thm]{Lemma}
\newtheorem{prop}[thm]{Proposition}
\theoremstyle{definition}
\newtheorem{defn}[thm]{Definition}
\newtheorem{remark}[thm]{Remark}
\newtheorem{example}[thm]{Example}

\newcommand{\KRStratum}[1]{\mathcal{A}_{I,#1}} 
\newcommand{\EOStratum}[1]{\mathrm{EO}_{#1}} 
\newcommand{\FinalElts}[1]{W _{\mathrm{#1,final}}} 
\newcommand{\FF}[1]{\mathbb{F}_{#1}} 
\newcommand{\Sp}[1]{\mathrm{Sp}_{#1}} 
\newcommand{\GSp}[1]{\mathrm{GSp}_{#1}} 
\newcommand{\SL}[1]{\mathrm{SL}_{#1}} 
\newcommand{\SU}[1]{\mathrm{SU}_{#1}} 
\newcommand{\relpos}[2]{\mathrm{inv}(#1,#2)} 
\newcommand{\Bor}{\mathcal{B}} 
\newcommand{\DL}[4]{X'_{#1} (#3,#4)} 
\newcommand{\DLB}[3]{X_{#1} (#2,#3)} 
\newcommand{\CoarseDL}[4]{X_{#1,#2} (#3,#4)} 
\newcommand{\HDR}[1]{H^{#1}_{\mathrm{DR}}} 

\newcommand{\scrF}{\EuScript{F}}
\newcommand\gfrac[2]{\genfrac{}{}{0pt}{}{#1}{#2}}

\begin{document}

\title[Ekedahl-Oort strata and Kottwitz-Rapoport strata]
{Ekedahl-Oort strata\\and Kottwitz-Rapoport strata}
\author{Ulrich G\"ortz}
\address[G\"ortz]{
Mathematisches Institut\\
Beringstr.~1\\
53115 Bonn\\
Germany}
\email{ugoertz@math.uni-bonn.de}
\thanks{G\"{o}rtz was partially supported by a Heisenberg grant and by the
SFB/TR 45 ``Periods, Moduli Spaces and Arithmetic of Algebraic Varieties''
of the DFG (German Research Foundation)}
\author{Maarten Hoeve}
\address[Hoeve]{
Korteweg-de Vries Instituut\\
Universiteit van Amsterdam \\
Plantage Muidergracht 24\\
1018 TV Amsterdam \\
The Netherlands}
\email{M.C.Hoeve@uva.nl}

\begin{abstract}
We study the moduli space $\mathcal A_g$ of $g$-dimensional principally
polarized abelian varieties in positive characteristic, and its variant
$\mathcal A_I$ with Iwahori level structure. Both supersingular
Ekedahl-Oort strata and supersingular Kottwitz-Rapoport strata are
isomorphic to disjoint unions of Deligne-Lusztig varieties (see
\cite{Hoeve2008} and \cite{GoertzYu2008}, resp.).  Here we compare these
isomorphisms. We also give an explicit description of Kottwitz-Rapoport
strata contained in the supersingular locus in the general parahoric case.
Finally, we show that every Ekedahl-Oort stratum is isomorphic to a
parahoric Kottwitz-Rapoport stratum.
\end{abstract}

\maketitle

\section{Introduction}

Consider the moduli space $\mathcal{A}_g$ of principally
polarized abelian varieties of dimension $g$ in characteristic $p$,
its variant $\mathcal{A}_I$ with Iwahori level structure at $p$,
and more generally the spaces $\mathcal{A}_J$ with parahoric level
structure of type $J$.
In these moduli spaces the supersingular locus is of great interest.
Whereas we have no hope to describe the whole moduli space
explicitly, for the supersingular locus one can be more optimistic.

One can also look at interesting closed subsets of the
supersingular locus. In the recent papers \cite{GoertzYu2008} and
\cite{GoertzYu2008-2} C.-F.~Yu and the first named author
explicitly described supersingular Kottwitz-Rapoport strata in
terms of Deligne-Lusztig varieties. In \cite{Hoeve2008} the second
named author did the same for supersingular Ekedahl-Oort strata.
See section \ref{sec2} for a summary of these results. Each EO stratum
$\EOStratum{w}$ admits a finite \'etale cover $\mathcal A_{I, w\tau} \rightarrow
\EOStratum{w}$ by a KR stratum (see \cite{ekedahl-vdgeer}, \cite{GoertzYu2008-2},
section 9). Our main results are:

\begin{enumerate}
\item The descriptions of supersingular KR-strata and supersingular EO-strata in 
terms of Deligne-Lusztig varieties are naturally compatible with respect to the projection
$\mathcal A_{I, w\tau} \rightarrow \EOStratum{w}$ (theorem \ref{thm:KR_EO_compatible}).

\item A KR stratum is contained in the supersingular locus of $\mathcal A_J$ if and only if 
it is superspecial (theorem \ref{ssp_iff_ssi}) and in that case it is isomorphic to a disjoint 
union of Deligne-Lusztig varieties (theorem \ref{parahoric_ssp_KR_as_DL}). 
This generalizes results in \cite{GoertzYu2008}, \cite{GoertzYu2008-2}.

\item Each EO-stratum is isomorphic to a certain KR-stratum in the parahoric moduli space 
$\mathcal A_J$ with $J$ the type of the canonical filtration (theorem \ref{EO_as_parahoric_KR}).
\end{enumerate}

Let us mention two open questions. First of all, which EO strata
occur in the image $\pi(\mathcal A_x)\subseteq \mathcal A_g$ for
$x \in {\rm Adm}(\mu)$? The answer to this question
could depend on $p$, but we don't expect that.
Ekedahl and van der Geer, who first posed
this question, showed in \cite{ekedahl-vdgeer}, that if $x=w\tau$
for a final element $w$, then the image is just the single EO
stratum corresponding to $w$, while in general, it is a union of
EO strata. Our theorem \ref{EO_as_parahoric_KR} says that for all
$y \in W_J w\tau W_J\cap {\rm Adm}(\mu)$, the image of $\mathcal
A_y$ under $\pi$ is equal to $\EOStratum{w}$.

Secondly, what is the dimension of the supersingular locus in moduli spaces
with a parahoric level structure? For general $g$ we only know this
for $J = \{0\}$, or $=\{g\}$, by the work of Li and Oort
\cite{li-oort}, or as a consequence of the purity of the Newton
stratification shown by de Jong and Oort, and for $J=I$ and $g$ even (see
\cite{GoertzYu2008-2}, also for bounds in the case where $g$ is odd).
We expect that  for $J\ne \{0\}, \{g\}$ the supersingular
locus is usually not equi-dimensional. The lower bound on the dimension
obtained from KR strata in the supersingular locus cannot be sharp in the
general parahoric case.
Note also that the union of supersingular EO strata achieves only
approximately half the dimension of the supersingular locus in $\mathcal
A_g$.

\section{Preliminaries and notation}
\label{sec2}
In this section we recall the most important definitions and results
from  \cite{GoertzYu2008}, \cite{GoertzYu2008-2} and \cite{Hoeve2008}.
Fix a prime $p$ and let $k$ be an algebraic closure of $\FF{p}$.
We only work with schemes over $k$.

\subsection{The symplectic group}
From the perspective of Shimura varieties the algebraic group
underlying the moduli spaces we study is the group $\GSp{2g}$ of
symplectic similitudes. We denote by $W_g$ its finite Weyl group,
generated by the simple reflections $s_1,\dots,s_g$, and by $W^a_g$
its affine Weyl group, generated by $s_0,s_1,\dots,s_g$. For $c
\leq g$ we see $W_c$ as a subgroup of $W_g$ via the natural map
induced by the inclusion of Dynkin diagrams, explicitly given by
$s_{c+1-i} \mapsto s_{g+1-i}$ for $i=1,\dots,c$.

We identify $I = \{ 0, \dots, g \}$ with the set of simple affine
reflections via $i\mapsto s_i$. The subsets of $I$ are the types
of parahoric subgroups of $\GSp{2g}(\mathbb Q_p)$. For $J \subset I$ we
let $W_J$ be the subgroup of $W_g^a$ generated by all $s_i$ with
$i\notin J$. \emph{Warning:} This differs from convention, where
$W_J$ is generated by $s_j$  for $j \in J$.

Let $\FinalElts{g} \subset W_g$ be the set of final elements, i.e., the set
of minimal length representatives for the cosets in $W_g/S_g$, where $S_g$
is the subgroup generated by $s_1, \dots, s_{g-1}$.
See \cite{oort01} and \cite{ekedahl-vdgeer}. The inclusion $W_c
\subset W_g$ maps $\FinalElts{c}$ to $\FinalElts{g}$. 

We write $G'$ for the inner form (over $\mathbb Q_p$) of
the derived group $\Sp{2g}$ of $\GSp{2g}$ that arises as the
automorphism group of a superspecial abelian variety (together
with a principal polarization), see \cite{GoertzYu2008}, Section
6.1. For $c \in \{0, \dots, [g/2] \}$ the subset $ \{ c, g-c
\} \subset I$ gives a parahoric subgroup $P'_{\{ c, g-c \}}$ of
$G'$. We denote the maximal reductive quotient of $P'_{\{ c, g-c \}}$
by $\overline{G}'_{c}$ (see loc.~cit., where this group is denoted by $\overline{G}'_{\{c,
g-c\}}$). This quotient is an algebraic group over $\mathbb
F_p$ which splits over $\mathbb F_{p^2}$. Its Dynkin diagram is
obtained from the extended Dynkin diagram of $\Sp{2g}$ by removing
the vertices $c$, $g-c$. Frobenius acts on the Dynkin diagram by
$i\mapsto g-i$. Denote by $\sigma'$ the Frobenius on $\overline{G}'_{c, k}$
of the non-split form $\overline{G}'_{c}$.

\subsection{Deligne-Lusztig varieties}

Let $G$ be a connected reductive group over $k$, defined over a
finite field $\FF{q}$ of characteristic $p$. Let $\sigma \colon G \to G$ be
its $q$-power Frobenius. Fix a Borel subgroup $B$ and a maximal torus in $B$ defined
over $\FF{q}$. Let $W$ be the Weyl group.

We denote with $\Bor(G) \cong G/B$ the variety of Borel subgroups
in $G$. In $\Bor(G)$ we have \emph{Deligne-Lusztig varieties}
\[
    \DLB{G}{w}{\sigma} (k) = \{ B\in \Bor(G)(k) \, | \,
    \relpos{B}{\sigma{B}} = w \}.
\]
Here $\mathrm{inv}$ is the relative position map. It sends
$(g_1,g_2) \in G/B \times G/B$  to the unique $w \in W$  such that
$g_1^{-1} g_2 \in BwB$.

To describe supersingular Ekedahl-Oort strata we also need \emph{fine
Deligne-Lusztig varieties} for parabolic subgroups of $\Sp{2c}$. For fine
Deligne-Lusztig varieties, rather than fixing only the relative position of
$P_0:=P$ and $\sigma(P)$, one fixes the relative positions of $P_n$ and
$\sigma(P_n)$ for all successive refinements $P_n = (P_{n-1}\cap
\sigma(P_{n-1}))R_u(P_{n-1})$, see \cite{Hoeve2008} section 2.2. To each $w \in
W_{c,\mathrm{final}}$ we attach a \emph{fine Deligne-Lusztig variety}
$\DL{\Sp{2c}}{I}{w}{\sigma}$ in the variety of stabilizers in $\Sp{2c}$ of
$c$-dimensional isotropic subspaces (denoted by $X(c)$ in
\cite{Hoeve2008}).

\subsection{Moduli spaces of abelian varieties}

Our main object of study is the moduli space $\mathcal A_g$ of
principally polarized abelian varieties of dimension $g$ over $k$.
To get a fine moduli space, we consider abelian varieties with a
full level $N$-structure for $N\geq 3$ and coprime to $p$, with respect to
a fixed primitive $N$-th root of unity. However, we
suppress the level structure from the notation, because it plays
only a minor role.

On $\mathcal A_g$ we have the \emph{Ekedahl-Oort stratification}
\[
    \mathcal A_g = \coprod_{w\in \FinalElts{g} } \EOStratum{w}.
\]
We denote
the supersingular locus in $\mathcal A_g$ by $\mathcal S_g$.

\subsection{Moduli spaces with parahoric level structure}
\label{sec2:moduli_spaces}

We also study variants of $\mathcal{A}_g$ with parahoric level
structure. Fix a subset $J \subseteq I$, or in other words the
type of a parahoric subgroup. We get a moduli space $\mathcal A_J$
with parahoric level structure of type $J$ at $p$, which roughly
speaking classifies chains $A_{j_0} \rightarrow A_{j_1}
\rightarrow \cdots \rightarrow A_{j_r}$ of isogenies (of fixed
$p$-power-degree) with $j_\nu\in J$. On each $A_j$ in the chain we have a
polarization $\lambda_j$; if $j=0$ or $j=g$, then this polarization is
principal. We denote by $\mathcal S_J$ the supersingular locus in $\mathcal
A_J$.

The \emph{Kottwitz-Rapoport stratification} of $\mathcal A_J$ is given as
\[
\mathcal A_J = \coprod_{x\in{\rm Adm}_J(\mu)} \mathcal A_{J, x},
\]
where ${\rm Adm}_J(\mu)$ is the image under $\widetilde{W}
\rightarrow W_J \backslash \widetilde{W} / W_J$ of the admissible
set ${\rm Adm}(\mu)$ (the set of all elements less than some
element in the $W$-orbit of $t^\mu$, where $\mu$ is the dominant
minuscule coweight $(1^{(g)}, 0^{(g)})$). For $x \in {\rm
Adm}(\mu)$, we denote by $\overline{x}$ its image in ${\rm
Adm}_J(\mu)$.

We have the natural projection $\pi_{J,I}\colon \mathcal A_I \rightarrow
\mathcal A_J$, and for each $x\in {\rm Adm}(\mu)$,
\begin{equation}\label{inverse_image_piJI}
\pi_{J,I}^{-1}(\mathcal A_{J, \overline{x}}) = \coprod_{v \in W_JxW_J \cap
{\rm Adm}(\mu)} \mathcal A_{I,v}.
\end{equation}
We identify $\mathcal A_g = \mathcal A_{\{0\}}$, and write $\pi$
for the projection $\mathcal A_I \rightarrow A_g$.

\subsection{Supersingular Kottwitz-Rapoport strata}
\label{subsec:ssKR}

We summarize the results of \cite{GoertzYu2008}.  Let $\tau$ be the minimal
element of $\mathrm{Adm} (\mu)$, and let $w$ be in $W_{\{c,g-c\}}$ so that
$\KRStratum{w \tau}$ is $c$-superspecial. Suppose that
$S$ is a connected scheme over $k$ and $y = (A_0 \to \dots \to
A_g,\lambda_0,\lambda_g,\eta)$ is an $S$-valued point of $\KRStratum{w
\tau}$.

Now consider the projection to $\mathcal A_{\{c,g-c\}}$. The point
$x = \pi_{\{c,g-c\},I} (y)$ is in $\pi_{\{c,g-c\},I}
(\KRStratum{\tau})$. Because $\KRStratum{\tau}$ is
$0$-dimensional, there are trivializations $A_c = S\times_k A'_c$
and $A_{g-c} = S \times_k A'_{g-c}$ with $A'_c$ and $A'_{g-c}$
both superspecial abelian varieties over $k$.

Let $\omega _i \subset\HDR{1}(A_i)$ be the Hodge filtration of
$A_i$ and $\omega_i '$ that of $A_i '$. We get flags
\[
0 \subsetneq \alpha (\omega _{j_1-1}) \subsetneq \alpha (\omega
_{j_1-2}) \subsetneq \dots \subsetneq \alpha (\omega _{j_0+1})
\subsetneq \omega_{j_0} / \omega_{j_1} = \mathcal{O}_S \otimes
(\omega_{j_0}' / \omega_{j_1}')
\]
where $(j_0,j_1)$ is equal to $(-c,c),(c,g-c)$ or $(g-c,g+c)$ and
we abusively write $\alpha$ for all the maps induced by $A_i \to
A_{j_0}$. These flags define a point $\phi (y)$ of $\Bor
(\overline{G}' _c)$.

\begin{thm}
\label{thm:KRStrata} For $w$ in $W_{\{c,g-c\}}$ the map $y \mapsto
(\pi_{\{c,g-c\},I}(y),\phi(y))$ is an isomorphism
\[
\KRStratum{w \tau} \stackrel{\sim}{\to} \pi_{\{c,g-c\},I}
(\KRStratum{\tau}) \times \DLB{\overline{G}'_{c}}{w^{-1}}{\sigma '}
.
\]
\end{thm}

\begin{proof}
This is corollary 6.5 in \cite{GoertzYu2008}. Also see the proof
of proposition 6.1 there for the definition of the morphism.
\end{proof}

Since $\pi_{\{c,g-c\},I} (\KRStratum{\tau})$ is a finite set of
points, $\KRStratum{w \tau}$ is a finite disjoint union of
Deligne-Lusztig varieties.

\subsection{Supersingular Ekedahl-Oort strata}
\label{subsec:ssEO}

Now we summarize the results of \cite{Hoeve2008}. In the case $g=1$, which
is more or less trivial from this point of view, one has to make some
obvious modifications, so we exclude it from the discussion.
Suppose that $w$
is in $\FinalElts{c} \subset \FinalElts{g}$ for some $c\leq g/2$.
Let $S$ be a connected scheme over $k$  and $x = (A,\lambda)$ be
an $S$-valued point of $\EOStratum{w}$.

Fix a supersingular elliptic curve $E$. There is a unique isogeny
\[
\rho \colon S \times E^g \to A
\]
such that the kernel of the restriction to each geometric point is
$\alpha _p ^c$ (\cite{Hoeve2008} theorem 1.2). The pull-back of
$\lambda$ gives a polarization $\mu$ on $E^g$. The pair
$(E^g,\mu)$ is a point of $\Lambda _{g,c}$, the set of isomorphism
classes of superspecial abelian varieties with a polarization with kernel
$\alpha_p ^{2c}$ (with level structure).

Let $\omega(-)$ denote the Hodge filtration of an abelian
variety. We get a subbundle
\[
\alpha (\omega (A)) \subsetneq \mathcal{O}_S \otimes (\omega(E^g)
/ \mu (\omega((E^g)^\vee))).
\]
In fact this is a bundle of $c$-dimensional isotropic subspaces in
a $2c$-dimensional symplectic vector space. Let $\psi (x)$ be the
stabilizer of this subbundle in $\Sp{2c}\times S$. This is an
$S$-valued point in a variety of parabolic subgroups of $\Sp{2c}$.

\begin{thm}
\label{thm:EOStrata} Suppose that $w$ is in $\FinalElts{c}$ for
some $c\leq g/2$ and $w \notin W_{c-1}$. Then the morphism $x
\mapsto ((E^g,\mu),\psi(x))$ is an isomorphism
\[
\EOStratum{w} \stackrel{\sim}{\to} \Lambda_{g,c} \times
\DL{\Sp{2c}}{J}{w^{-1}}{\sigma ^2}.
\]
\end{thm}

\begin{proof}
This is theorem 1.2 in \cite{Hoeve2008}, except for a few
superficial differences.

First of all in \cite{Hoeve2008} lemma 3.1 the isomorphism is
constructed with Dieudonn\'e theory. If $M(A[p])$ is the
Dieudonn\'e module of $A[p]$, then there is a natural isomorphism
$M(A[p])^{(p)} \cong H^1 _\mathrm{dR} (A)$. The Verschiebung gives
an isomorphism
\[
    \omega(E^g) / \mu (\omega((E^g)^\vee))
    \stackrel{\sim}{\leftarrow}
        M(E^g [p]) / \mu( (E^g)^\vee[p] ) = M(\ker (\mu)).
\]
Under this isomorphism the image  $\omega (A)$ corresponds to the
image of $M(A[p])$,  to which $x$ is sent in \cite{Hoeve2008}. So
the two constructions are equivalent.

Secondly, we get $w^{-1}$ instead of $w$, because we use a different
indexation. See \cite{GoertzYu2008-2} section 2.4 for the different indexations
for the EO-stratifications. In \cite{Hoeve2008} the indexation of 
Moonen and Wedhorn is used, while in \cite{GoertzYu2008} the indexation
of van der Geer is used. 

Thirdly, here we don't need to divide out the action of an
automorphism group. This action is trivial because of the level
structures.
\end{proof}

\section{Comparison between supersingular EO strata and supersingular KR
strata}
\label{sec3}

Remember that for $w\in \FinalElts{g}$ the projection $\pi \colon
\mathcal{A}_{g,I} \to \mathcal{A}_g$ restricts to a finite \'etale
surjective map $\KRStratum{w \tau} \to \EOStratum{w}$, see
\cite{ekedahl-vdgeer}, \cite{GoertzYu2008-2} section 9. In this section we
show how this map relates to the above descriptions of
supersingular KR- and EO-strata. We start with the connected
components.

\begin{lem}
The map $\beta_c\colon \pi_{\{c,g-c\}} (\KRStratum{\tau}) \to
\Lambda_{g,c}$ that sends a point $((A_c,\lambda _c) \to (A_{g-c},\lambda
_{g-c}))$ to $(A_{c}^\vee,\lambda _{c}^\vee)$ is a bijection.
\end{lem}

\begin{proof}
We construct an inverse as follows. Given
$(A,\lambda)\in \Lambda_{g,c}$ we put $A_c = A^\vee$, $A_{g-c} = A^{(p)}$.
The Frobenius $A \rightarrow A^{(p)}$ factors through $A^\vee$, and we get
an isogeny $A_c \rightarrow A_{g-c}$. This, together with the natural
polarizations on $A_c$, $A_{g-c}$, is the desired point in $\pi_{\{c,g-c\}}
(\KRStratum{\tau})$.
\end{proof}

Next we look at the index sets of the stratifications.

\begin{lem}
We have:
\[
\FinalElts{g} \cap W_{\{c,g-c\}} = \FinalElts{c}.
\]
\end{lem}

\begin{proof}
Since $W_c \subset W_{\{c,g-c\}}$, the right hand side is included
in the left hand side.

For the other inclusion, suppose that $w$ is in $\FinalElts{g}
\cap W_{\{c,g-c\}}$. The set of simple reflections in a reduced
expression of $w$ is equal to $ \{s_i, s_{i+1},\dots ,
s_{g-1},s_g\}$ for some $i$ (see the proof of lemma 7.1 in
\cite{Hoeve2008}). Since $s_{g-c}$ is not in $W_{\{c,g-c\}}$, we
must have $i>g-c$. Because $W_c \subset W_g$ is generated by all
$s_j$ with $j>g-c$, it contains $w$.
\end{proof}

Finally we look at the Deligne-Lusztig varieties. The group
$\overline{G}' _c$ splits over $k$ (in fact already over
$\FF{p^2}$) as
\[
    \overline{G}'_c \cong \Sp{2c} \times
    \SL{g-2c} \times \Sp{2c}.
\]
as can be seen  from the Dynkin diagram of $\overline{G}' _c$.
It follows from the description of $\sigma '$ in
\cite{GoertzYu2008} that with respect to this decomposition
\[
\sigma ' (g_1,g_2,g_3) =
(\sigma(g_3),\tilde{\sigma}(g_2),\sigma(g_1)),
\]
where $\sigma$ is the Frobenius of $\mathrm{Sp}_{2c}$ and
$\tilde{\sigma}$ is the Frobenius of a $\SU{g-2c}$. 

Anything related to $\overline{G}' _c$ splits in a similar way.
For instance, its absolute Weyl group splits as
\begin{equation} \label{Weyl_gp_as_product}
W_{\{c,g-c\}} = W_c \times S_{g-2c} \times W_c,
\end{equation}
where $S_{g-2c}$ is the symmetric group, which is the Weyl group
of $\SL{g-2c}$. The inclusion $W_c \subset W_g$ we are using, is
the inclusion of $W_c$ in $W_{\{c,g-c\}}$ as the last factor above, followed by
$W_{\{c,g-c\}} \subset W_g$.

\begin{prop}
For $w\in W_c \subset W_g$, the projection
\[
    \Bor(\overline{G}' _c) = \Bor(\Sp{2c})\times \Bor(\SL{g-2c})
    \times \Bor(\Sp{2c}) \to \Bor(\SL{g-2c})
    \times \Bor(\Sp{2c})
\]
to the final two factors induces an isomorphism
\[
    \DLB{\overline{G}'_c}{w}{\sigma'} \cong
    \DLB{\SL{g-2c}}{1}{\tilde{\sigma}} \times
    \DLB{\Sp{2c}}{w}{\sigma  ^2}.
\]
\end{prop}

\begin{proof}
The Deligne-Lusztig variety $\DLB{\overline{G}' _c}{w}{\sigma '}$
consists of triples $(g_1,g_2,g_3)$ such that
\[
\relpos{g_1}{\sigma(g_3)} = 1, \,
\relpos{g_2}{\tilde{\sigma}(g_2)} = 1, \,
\relpos{g_3}{\sigma(g_1)} = w'.
\]
The first and last equations are equivalent to $g_1 = \sigma
(g_3)$ and $\relpos{g_3}{\sigma ^2 (g_3)} = w'$, which is the
equation for $\DLB{\Sp{2c}}{w}{\sigma  ^2}$.
\end{proof}

The proposition gives a morphism
\[
f \colon \DLB{\overline{G}'_c}{w}{\sigma'} \to
    \DLB{\Sp{2c}}{w}{\sigma  ^2} \to
    \DL{\Sp{2c}}{I}{w}{\sigma ^2}.
\]

\begin{thm}\label{thm:KR_EO_compatible}
For $w \in \FinalElts{c}$ there is a commutative diagram
\[
\begin{diagram}
\KRStratum{w \tau} & \rTo ^{\sim} &  \pi_{\{c,g-c\}}(\mathcal A_{I,\tau}) \times
\DLB{\overline{G}'_c}{w^{-1}}{\sigma '} \\
\dTo ^{\pi} & & \dTo _{\beta_{c} \times f} \\
\EOStratum{w} & \rTo ^\sim &  \Lambda_{g,c} \times
\DL{\Sp{2c}}{I}{w^{-1}}{\sigma ^2}
\end{diagram}
\]
where the horizontal maps are the isomorphisms from
\ref{thm:KRStrata} and \ref{thm:EOStrata}.
\end{thm}

\begin{proof}
Suppose $(A_i)_{i\in I}$ is an $S$-valued point  in $\KRStratum{w \tau}$.
Since the morphism $\rho \colon S \times E^g \to A_0$ from theorem
\ref{thm:EOStrata} is unique, it must be equal to $A_{-c} \to
A_0$. (Here we extend the chain $(A_i)_i$ to a chain indexed by $\mathbb Z$
by duality and periodicity as usual.)
Under the horizontal map in the upper row, $(A_i)_i$ is mapped to the flags
of the images of the Hodge filtrations $\omega(A_i)$ inside $\omega(A_c)/\omega(A_{-c})$,
$\omega(A_{g-c})/\omega(A_{c})$, and $\omega(A_{g+c})/\omega(A_{g-c})$.
Because the stratum is $c$-superspecial, these quotients can actually be
identified with quotients of the de Rham cohomology, or in other words of
lattices in the standard lattice chain, however with a shift of the indices
by $g$. See the proof of Prop.~6.1 in \cite{GoertzYu2008}. This shows that
when we identify the Weyl group of $\overline{G}'_c$ with $W_{\{c, g-c\}}$
and decompose it as in \eqref{Weyl_gp_as_product}, then the subgroup
generated by $s_0, \dots, s_{c-1}$ corresponds to the flags in
$\omega(A_{g+c})/\omega(A_{g-c})$; the middle factor, i.e., the subgroup
generated by $s_{c+1}, \dots, s_{g-c-1}$ corresponds to the flags in
$\omega(A_{g-c})/\omega(A_{c})$; and finally the part generated by
$s_{g-c+1}, \dots, s_g$ corresponds to flags in
$\omega(A_c)/\omega(A_{-c})$. The projection considered in the previous
proposition projects onto the latter two factors, and in particular, the
map $f$ we defined takes the flag $(\omega(A_i))_i$ to $\omega(A_0)$ which is
the maximal totally isotropic member of the flag in the third factor.
Since the map $\pi$ which gives the left column maps our chain $(A_i)_i$ to
$A_0$, and the lower horizontal map maps $A_0$ to $\omega(A_0)$, we see
that the diagram commutes.
\end{proof}


\section{Supersingular KR strata in the parahoric case}
\label{sec4}

Let $J \subseteq I$ be the type of a standard parahoric subgroup. We call
a KR-stratum in $\mathcal A _J$ \emph{supersingular}, if
it is contained in the supersingular locus.  We want
to describe these  strata. As one would hope, there is a description completely
analogous to the one in \cite{GoertzYu2008}, and it can in fact be
derived from the Iwahori version with relatively little additional work.

We make the following definitions, analogous to the Iwahori
case.

\begin{defn}
Let $J \subseteq I$, and let $c\in I$.
\begin{enumerate}
\item
A stratum $\mathcal A_{J, x}$ is called \emph{$c$-superspecial},
if $\pi_{J, I}^{-1}(\mathcal A_{J, x})$ is a union of
$c$-superspecial KR strata in $\mathcal A_I$, i.e., if
\[
W_J x W_J \cap {\rm Adm}(\mu) \subseteq W_{\{c, g-c\}}\tau.
\]
\item
A stratum $\mathcal A_{J, x}$ is called \emph{superspecial}, if it
is superspecial for some $c\in I$.
\end{enumerate}
\end{defn}

\begin{lem}
There exists a $c$-superspecial stratum in $\mathcal A_J$ if and only if
$c$ and $g-c$ are in $J$.
\end{lem}

\begin{proof}
If $c, g-c \in J$ then $W_J\tau W_J \subseteq W_{c, g-c}\tau$ which shows
that the minimal KR stratum in $\mathcal A_J$ is $c$-superspecial.
To show the converse, assume that $c\not\in J$ or $g-c\not\in J$. Because
the union of all $c$-superspecial strata is closed, it is enough to show
that the minimal stratum in $\mathcal A_J$ is not $c$-superspecial,
i.e., that
\[
W_J\tau W_J \cap {\rm Adm}(\mu) \not\subseteq W_{\{c, g-c\}}\tau.
\]
By assumption we have $s_c \tau, s_{g-c} \tau = \tau s_c \in W_J\tau W_J$.
The lemma follows, because for all simple reflections $s$ we have $s\tau
\in {\rm Adm}(\mu)$. This can be shown by checking that all these elements
are permissible. Another (much more roundabout) way to prove this is to say
that for every $c$, either $s_c\tau \le t^\mu$ or $s_{g-c}\tau\le t^\mu$,
since clearly $\mathcal A_{I,t^\mu}$ is not superspecial. The claim follows
because conjugating by $\tau$ preserves the Bruhat order, interchanges
$s_c$ and $s_{g-c}$ and maps $t^\mu$ to $t^{w_\emptyset\mu}$, where
$w_\emptyset$ is the $W$-component of $\tau \in \widetilde W = X_*(T)
\rtimes W$.
\end{proof}

In view of \eqref{inverse_image_piJI}, the lemma says that for
$x\in {\rm Adm}(\mu)$ and $c\not\in J$, we can find a point
$(A_j)_{j\in J}\in \mathcal A_{J, \overline{x}}$ which can be
extended to a chain $(A_j)_{j\in I}$ which does not lie in the
$c$-superspecial locus, i.e., where we do not have $A_c$,
$A_{g-c}$ superspecial and $A_c \rightarrow A_{g-c}^\vee$ the
Frobenius.

We now describe $c$-superspecial KR-strata in terms of
Deligne-Lusztig varieties. Suppose that $J$ is a subset of $I$
with $c,g-c \in J$ (so that $c$-superspecial strata exist in
$\mathcal A _J$). Let $\mathcal P _{\sigma'(J)} (G'_{c,k})$  be the variety
of parabolic subgroups of type $I\backslash \sigma'(J)$ in $G'_{c,k}$. Note
that it is not defined over $\mathbb{F}_p$ if $J$ is not Frobenius
invariant. There are coarse Deligne-Lusztig varieties
\[
    \CoarseDL{G'_{c,k}}{\sigma'(J)}{\overline{w}}{\sigma'} = \{ P \in
    \mathcal P _{\sigma'(J)}
    (G'_{c,k}) \, | \, \relpos{P}{\sigma' (P)} = \overline{w} \}
\]
for all double cosets $\overline{w} \in W_{\sigma'(J)} \backslash
W_{\{c,g-c\}} / W_J$.

\begin{thm}\label{parahoric_ssp_KR_as_DL}
Let $\mathcal A _{J,\overline{x}}$ be a $c$-superspecial stratum inside
$\mathcal A_J$. Write $x = w \tau$, so that $w\in W_{\{c,g-c\}}$.
There is an isomorphism
\[
\mathcal A_{J, \overline{x}} \cong \pi_{\{c, g-c\},
I}(\mathcal A_{I, \tau}) \times
\CoarseDL{G'_{c,k}}{\sigma'(J)}{\overline{w^{-1}}}{\sigma'}.
\]
\end{thm}

\begin{proof}
Taking into account the remarks above, the proof is the same as in the
Iwahori case, see \cite{GoertzYu2008}, section 6.
\end{proof}

It is evident that every superspecial stratum is supersingular. We
will show below that the converse is true, as well.
We first prove a connectedness result in the parahoric
case, analogous to \cite{GoertzYu2008-2}, Thm.~7.3.

\begin{prop}
If a KR-stratum is not superspecial, then it is irreducible.
\end{prop}

The converse holds if the level structure is large enough.

\begin{proof}
Let $x\in {\rm Adm}(\mu)$ and suppose $\mathcal A_{J,\overline{x}}$ is not
superspecial, i.~e.
\[
W_J x W_J \cap {\rm Adm}(\mu) \not\subseteq W_{\{c,g-c\}}\tau,
\quad\text{for all } c \in \{ 0, \dots, [g/2] \}.
\]
Then there exists for each $c \in \{ 0, \dots, [g/2] \}$ an
element in $W_J x W_J \cap {\rm Adm}(\mu)$ larger than $s_c\tau$ or
than $s_{g-c}\tau$.

Therefore by \cite{GoertzYu2008-2} Thm.~7.2 the closure of the union of all
$1$-dimensional strata in $\pi_{J,I}^{-1}(\mathcal A_{J, \overline{x}})$,
\[
\overline{\bigcup_{v \in W_JxW_J \cap {\rm Adm}(\mu)} \bigcup_{\gfrac{s \le
v\tau^{-1}}{\ell(s)=1}}  \mathcal A_{I, s}}
\]
is connected. Now every connected component of
\[
\pi_{J,I}^{-1}(\overline{\mathcal A_{J, \overline{x}}})
= \overline{\bigcup_{v \in W_JxW_J \cap {\rm Adm}(\mu)} \mathcal A_{I, v}},
\]
meets the previous set (because every connected component of a KR stratum
in $\mathcal A_I$ has a point of the minimal stratum in its closure,
\cite{GoertzYu2008-2}, Thm.~6.2).
This implies that $\pi_{J,I}^{-1}(\overline{\mathcal A_{J, \overline{x}}})$
is connected, but then its image $\overline{\mathcal A_{J,
\overline{x}}}$ is connected, too. This closure is normal, so connectedness
implies irreducibility, and the proposition follows.
\end{proof}

\begin{thm} \label{ssp_iff_ssi}
A KR-stratum is supersingular if and only if it is superspecial.
\end{thm}

\begin{proof}
Suppose that $\mathcal A_{J, \overline{x}}$ is contained in
$\mathcal S_J$. Its image in $\mathcal A_g$ is a union of
\emph{supersingular} EO strata. We may assume that the level
structure outside $p$ is large enough, so that this union, and
hence $\mathcal A_{J, \overline{x}}$, is reducible. By the
previous proposition, this implies the desired statement.
\end{proof}

\begin{remark}
Let us discuss the combinatorial meaning of the theorem. Start with $x\in
{\rm Adm}(\mu)$, such that $\mathcal A_{J, x} \subseteq \mathcal S_J$. The
latter condition is equivalent to
\[
\pi_{J,I}^{-1}(\mathcal A_{J, \overline{x}}) \subseteq \mathcal
S_J,
\]
or in other words to
\[
W_JxW_J \cap {\rm Adm}(\mu) \subseteq \bigcup_{c=0}^{[g/2]} W_{\{c,g-c\}}.
\]
By the theorem there exists $c\in J$ such that
\[
W_JxW_J \cap {\rm Adm}(\mu) \subseteq  W_{\{c,g-c\}}
\]
It seems hard to prove this statement combinatorially because it is not
easy to understand what happens when one intersects the double coset
$W_JxW_J$ with the admissible set.
\end{remark}


\section{EO strata as parahoric KR strata}
\label{sec5}

In this section, we fix a final element $w\in W_{g, {\rm final}}$, and denote by
$J\subseteq  I$ the type of the corresponding canonical filtration. Since
the canonical filtration is a flag in $\mathbb H$ (rather than a lattice
chain), we have $0\in J$ ``automatically''. Furthermore, since the
canonical filtration always contains a maximal totally isotropic subspace,
we also have $g\in J$. Furthermore, $J$ is $\sigma$-stable.

\begin{example}
On the EO-stratum of abelian varieties with $a$-number 1 and
$p$-rank $f$, the type of the canonical filtration is
$\{0,f,f+1,\dots,2g-f-1,2g-f,2g\}$, see \cite{ekedahl-vdgeer}
example 3.4. On the stratum of abelian varieties with $a$-number
$a$ and $p$-rank $g-a$ the canonical type is $\{0,g-a,g,g+a,2g\}$.
\end{example}

\begin{example}
In the case $g=2$, we have four final elements, corresponding to the
superspecial locus (${\rm id}$), the supersingular locus without the superspecial
points ($s_2$), the $p$-rank $1$ locus  ($s_1s_2$) and the $p$-rank $2$
locus ($s_2s_1s_2$). The corresponding canonical filtrations are given by
\[
J = \{ 0, 2 \}, \quad
 \{ 0, 1, 2 \}, \quad
 \{ 0, 1, 2 \}, \quad
 \{ 0, 2 \},
\]
respectively.
\end{example}

In the next theorem, we view $\mathcal A_J$ as the moduli
space of principally polarized abelian schemes $A$ together with a flag
$\scrF_\bullet$ of finite
flat subgroup schemes of $A[p]$ of appropriate ranks, as determined by $J$.

From the flag $\scrF_\bullet$ we can make a conjugate flag $\scrF_
\bullet ^c$ by
\[
\scrF_{g+i} ^c = V^{-1}(\scrF_i ^{(p)}) \quad \textrm{and} \quad
\scrF_{g-i} ^c = (\scrF_{g-i} ^c)^\bot \quad (i=1,\dots,g).
\]
The stratum $\mathcal A_{J, \overline{w\tau}}$ consists of all
points $(A, (\scrF_j)_{j\in J})$ such that $\scrF_\bullet$ and
$\scrF_ \bullet ^c$ (or rather their Dieudonn\'e modules) are in
relative position $w$.

\begin{thm}
\label{EO_as_parahoric_KR} The natural map $\pi_J \colon \mathcal
A_J \to \mathcal A _g$ restricts to an isomorphism $\mathcal A_{J,
\overline{w\tau}} \to \EOStratum{w}$. Its inverse maps a point
$A\in \EOStratum{w}$ to $(A, \scrF_\bullet^{\rm can} \subset A[p])$, where
$\scrF_\bullet^{\rm can}$ denotes the canonical filtration of
$A[p]$.
\end{thm}

\begin{proof}
Because over $\EOStratum{w}$ the canonical filtration can be constructed
globally (see \cite{oort01}, Prop.~(3.2)) we have a section $s
\colon  \EOStratum{w} \to \mathcal A_J$ of $\pi_J$. By definition
its image is in $A_{J, \overline{w\tau}}$.

To finish the proof, we show that for each point $(A,
(\scrF_j)_{j\in J}) \in \mathcal A_{J, \overline{w\tau}}(k)$,
$(\scrF_j)$ is the canonical filtration of $A[p]$. Let $\nu$ be
the final type of $w$.  We know that for all points in the image
of $s$ we have $V(\scrF_j) = \scrF_{\nu(j)}$. But then this must
hold for all points in $\mathcal A_{J, \overline{w\tau}}(k)$,
since the relative position of $V(\scrF_\bullet)$ and
$\scrF_\bullet$ is constant on $\mathcal A_{J,
\overline{w\tau}}(k)$. This implies that $\scrF_\bullet$ is
the canonical filtration.
\end{proof}

There is a commutative diagram
\[
\xymatrix{
    \mathcal A_{I, w\tau} \ar[rr]^{\pi_{J,I}} \ar[rd]_{\pi} && \mathcal A_{J,
    \overline{w\tau}} \ar[dl]^\cong \\
    & \EOStratum{w} &
}
\]
So for $w$ final and $J$ as above the horizontal map is surjective.
Since all KR strata in
$\pi_{J, I}^{-1}(\mathcal A_{J, \overline{w\tau}})$ map
(necessarily surjectively) to $\EOStratum{w}$, $w\tau$ is the unique
element of minimal length in $W_Jw\tau W_J$. So we are in a very
special situation.

\begin{remark}
In case $\EOStratum{w}$ is supersingular, i.e., contained in $\mathcal
S_g$, we have theorems \ref{thm:EOStrata} and \ref{parahoric_ssp_KR_as_DL}
at our disposal, and the isomorphism above corresponds to the
identification of fine Deligne-Lusztig varieties with coarse
Deligne-Lusztig varieties for a different parabolic subgroup, see
\cite{Hoeve2008} Cor.~2.7.  Also, from \cite{GoertzYu2008} theorem 6.3 we
know that the fibres of $\pi$ and of the horizontal map in the diagram are
Deligne-Lusztig varieties.
\end{remark}



\begin{thebibliography}{99}
\bibitem{ekedahl-vdgeer} T.~Ekedahl, G.~van der Geer, Cycle classes of the
E-O stratification on the moduli of abelian varieties,
arXiv:math.AG/0412272v2.

\bibitem{GoertzYu2008} U.~G\"ortz, C.-F.~Yu, Supersingular Kottwitz-Rapoport
strata and Deligne-Lusztig varieties,
arXiv:0802.3260v2.

\bibitem{GoertzYu2008-2} U.~G\"ortz, C.-F.~Yu, The supersingular locus in
Siegel modular varieties with Iwahori level structure,
arXiv:0807.1229v1.

\bibitem{Hoeve2008} M.~Hoeve, Ekedahl-Oort strata in the supersingular locus,
arXiv:0802.4012v1.

\bibitem{kottwitz-rapoport:alcoves} R. E.~Kottwitz and M. Rapoport,
  Minuscule alcoves for ${\rm GL}_n$ and ${\rm GSp}_{2n}$. manuscr.~math.~{\bf
  102} (2000), 403--428.

\bibitem{li-oort} K.-Z.~Li and F.~Oort, \emph{Moduli of Supersingular Abelian
Varieties}, Lecture Notes in Mathematics \textbf{1680}, Springer (1998).

\bibitem{oort01} F.~Oort, A stratification of a moduli space of abelian
varieties, in: \emph{Moduli of abelian varieties (Texel Island, 1999)},
Progr.~Math.~{\bf 195}, 345--416, Birkh\"auser 2001.

\end{thebibliography}
\end{document}